\documentclass[a4paper, 11pt, leqno]{article}
\usepackage{amssymb, amsmath, a4wide, enumerate, theorem, latexsym, amsfonts,  pifont, stmaryrd}
\usepackage[english]{babel}
\usepackage{multirow}
\usepackage{apalike, soul}
\usepackage{bbding}
\usepackage{pifont}
\usepackage{setspace, comment, lingmacros}
\usepackage[margin=2.5cm,nohead]{geometry}
\usepackage[all]{xy}
\usepackage{bm}
\usepackage{verbatim}
\usepackage[normalem]{ulem}
\usepackage{xcolor}

\usepackage{fancyhdr}
\usepackage{appendix}
\usepackage{epigraph}
\usepackage{boxedminipage}
\usepackage{comment}
\usepackage{graphicx}

\newcommand{\Her}{\textrm{Her}}

\theorembodyfont{\rmfamily}
\newtheorem{Thm}{Theorem}
\newtheorem{Def}{Definition}
\newtheorem{Lem}{Lemma}

\newtheorem{Prop}{Proposition}

\newcommand{\marg}[1]{\marginpar{\scriptsize \singlespace \fbox{\parbox{18mm}{#1}}}}

\newcommand{\stb}{\textsc{Stb}}

\newcommand{\set}{\textsc{Set}}

\newcommand{\ed}{\textrm{ED}}

\def\L{\mathcal{L}}

\begin{document}

\author{\O ystein Linnebo }
\title{Strict potentialism in modal mirrors}
\date{Version of 6 March 2026}

\hyphenation{quanti-fi-ca-tion}

\maketitle

\begin{abstract}
    \noindent Potentialism is the view that objects are successively generated in an incompletable process. A strict version of the view adds that truths are successively determined. Strict potentialism can be analyzed using two modalities: one for the generation of objects, another for truths becoming determined. The result is a classical bimodal logic. We obtain simpler and more user-friendly theories by invoking so-called mirroring theorems to ``switch off'' one or both modalities, in return for a less classical logic. When the modality of object generation is switched off, we obtain a restricted plural logic. When the modality of truth determination is switched off, the logic becomes intuitionistic. Finally, the value of this general approach to strict potentialism is illustrated by applications to a Weyl-inspired predicative set theory, Cantor's domain principle, and strict potentialism about Cantorian sets. 
\end{abstract}

\begin{spacing}{1.5}

\section{Introduction} 


Cantor famously defines the notion of set as follows:
\begin{quote}
By a `set' we understand every collection into a whole $M$ of determinate, well-distinguished objects $m$ of our intuition or our thought (which will be called the `elements' of $M$). We write this as: $M = \{m\}$. 
(Cantor 1895, 481) 
\end{quote}
The idea is that a set is obtained from many objects $m$. Using the notation of contemporary plural logic, we may designate these many objects as `$mm$'. By collecting them ``into a whole'', we obtain a single set $M = \{mm\}$. 

What about the set-theoretic paradoxes? 
Cantor was forced to comment on this matter only two years later when the paradoxes were discussed among German mathematicians. In an 1897 letter to Hilbert, he wrote: 
\begin{quote}
    I say of a set that it can be thought of as finished [\ldots] 
    if it is possible without contradiction (as can be done with finite sets) to think of all its elements as existing together, and so to think of the set itself as \emph{a compounded thing for itself}; or (in other words) if it is \emph{possible} to imagine the set as actually existing with the totality of its elements. \cite[p.~927]{Ewald:1996}
\end{quote}
The letter continues: 
\begin{quote}
    And so too in the first article of [Cantor 1895---quoted above], I define `set' (meaning thereby only the finite or transfinite) at the very beginning as an `assembling together' [Zusammenfassung]. But an `assembling together' is only possible if an `\textit{existing together}' [\textit{Zusammensein}] is \textit{possible}. (928)
\end{quote}
There is a simple and attractive idea here. Any multiplicity of objects that ``exist together'' can, according to Cantor ``be gathered together into `\textit{one} thing'\,'', namely, their set. 

With this example in place, let me fix some terminology. 
When a single set $M$ is fully characterized in terms of its many elements $mm$, we have a \textit{generative definition} of $M$, which is thus \textit{generated} from $mm$. 
Since these generative definitions can be iterated, we may talk about a \textit{process} of generation. Thus, this process need not take place in time but can be understood abstractly. A generative process can be \textit{completable} or not. Cantor famously regards the generation of natural numbers as completable. When completed, all the natural numbers ``exist together'', such that their set too can be generated. 
Next, let \textit{potentialism} be the view that certain kinds of mathematical object are successively generated in an incompletable process. An austere version is Aristotle's view that no infinite collection---not even the natural numbers---is completable. A far more relaxed version is Cantor's view that, although \textit{each} set is completable, the collection of \textit{all} sets is not: for all sets cannot exist together.  

Potentialism is often explicated modally. 
Let us use plural variables to represent multiplicities of objects that exist together (what Cantor elsewhere calls ``consistent multiplicities''). Then it is natural to take Cantor to be committed to the following view: 
\begin{equation}\label{eq:pot-set-E}
    \Box \forall xx \Diamond \exists y \textsc{Set}(xx,y)
\end{equation}
where `$\textsc{Set}(xx,y)$' states that $y$ is the set whose elements are $xx$. 
As Cantor is well aware, though, it is impossible to conceive of the multiplicity of all sets as ``one finished thing''. This naturally corresponds to the claim that it is impossible for the domain to be closed under the transition from some coexisting objects to the corresponding set:
\begin{equation} \label{eq:no-tot-sets}
    \neg \Diamond \forall xx \exists y \textsc{Set}(xx,y)
\end{equation}

Two definitions remain. Let \textit{actualism} be the view that finds no place for potentiality or incompletability in mathematics. All the sets do ``exist together'', just like the books in my office. 
It is just that the sets are too many or otherwise unsuitable for the operation of set formation to be applied to them. 
Finally, \textit{strict potentialism} is a more thoroughgoing form of potentialism. Strict potentialists add to the core potentialist claim about the successive generation of \textit{objects} an analogous claim concerning \textit{truths}, namely, that truths too are successively determined or ``made true'' in an incompletable process.\footnote{This concept of strict potentialism was introduced in \cite[pp.~83--4]{Linnebo:2017-PhilMaths} and further analyzed in \cite{Linnebo&Shapiro:ActPotInf}; see also  \cite{Crosilla-Linnebo:2024-Weyl}, \cite{Linnebo&Shapiro:2023-PredFormPot}, and \cite{Linnebo&Shapiro:2024-PredClassesStrictPot}.} 

How are we to understand this notion of truths being determined or made true? Here are some motivating examples. Suppose the set $M = \{mm\}$ is generated at some stage. At this stage it is determined, for each member of $mm$, that it is an element of~$M$. It is also determined that no other object, whether already available or yet to be generated, is an element of $M$: $\forall x (x \in M \to x \prec mm)$. Plausibly, the law of extensionality for sets is determined right at the beginning of the process of set generation; for our notion of set ensures that any set ever generated must be subject to the law of extensionality. 
The notion of determination relies on an idea of \textit{intrinsic} truthmaking (or, for that matter, grounding). Each truth owes its truth to facts that are available at some stage of the generative process. The truth holds in virtue of these available facts, with no reliance on any facts located higher in the hierarchy. 


The purpose of this article is to explore and compare some different logical explications of strict potentialism and thus to gain a better understanding of the view and its logical and mathematical consequences. 

Since we need to represent both the generation of objects and the determination of truths, it is natural to attempt a bimodal explication. First, there is a modality of generation (``G-modality''). Given some objects $mm$, for example, it is possible to generate their set $M = \{mm\}$. We represent this by means of a G-accessible world that contains the set $M$. Second, there is a modality of determination (``D-modality''). When $\varphi$ is made true (or grounded) at a stage $w$, we say that $\varphi$ is \textit{determined at $w$} and represent this as $\varphi$ being true at all worlds that are D-accessible from $w$. 

This may seem overwhelming. Even a single modality in the analysis of mathematics is unconventional. The use of \textit{two} modalities will strike some people as simply too much. I sympathize. I will therefore study ways to eschew modal operators. A central achievement will be to show that each modality can be absorbed into the interpretation of the connectives and quantifiers, thus allowing us to drop the corresponding modal operator from the syntax. I will refer to this maneuver of taking a modal operator out of the syntax, with a commensurate adjustment of the interpretation of the logical operators, as ``switching off'' the modality. The reverse operation ``switches on'' the modality. 

This ability to switch a modality on and off, and the logical consequences of doing so, will be examined. Our analysis will reveal a trade-off between two unconventional steps. Strict potentialists need to accept \textit{either} the use of modality in mathematics \textit{or} a departure from the classical or traditional logic. With generative (or G-) modality  switched on, we can use traditional plural logic. When it is switched off, we must use a more restricted form of plural logic that licenses fewer pluralities. When the modality of determination (D-modality) is switched on, we can use classical propositional logic. When it is switched off, we must use intuitionistic logic. 
The following table summarizes these results: 

\vspace{2mm}
\begin{center}
\begin{tabular}{ccc}
  \hline
  \textit{modality on or off} && \textit{logical effect} \\ \hline \hline
  G-modality on & &traditional plural logic \\
  G-modality off & &restricted plural logic \\ 
  D-modality on & &classical logic \\
  D-modality off && intuitionistic logic \\
\end{tabular}
\end{center}

\vspace{2mm}
It is useful to be able to switch the modalities on and off. For the purposes of philosophical analysis, it is often helpful to switch them on, to obtain a maximally explicit description of the relevant phenomenon. This also makes available classical first-order logic and traditional plural logic. But for the purposes of doing strict potentialist mathematics or comparing with familiar views in the foundations of mathematics, it is simpler and more user-friendly to switch the modalities off.\footnote{This work completes a trio of articles on how the modalities of potentialism can be switched on or off: \cite{Linnebo:PotDemodal} uses a critical plural logic to ``demodalize'' liberal potentialism, while \cite{Linnebo:2027-WhatIsPot} discusses the philosophical significance of this demodalization.} 

Here is the plan. First, I provide a bimodal explication of strict potentialism (\S\ref{sec:bimodal}). The ``distance'' between the two modalities provides a measure of the non-classicality to which a form of strict potentialism is committed (\S\ref{sec:amount-indet}). Next, I prove theorems that enable us to switch the two modalities on and off---first, for first-order logic (\S\S\ref{sec:FOL-mirroring}--\ref{sec:bi-mirroring-FOL}); then, for plural logic (\S\S\ref{sec:mono-mirroring-PL}--\ref{sec:bi-mirroring-PL}). I observe, in passing, how these results give rise to a potentialist version of Kripke semantics for intuitionistic (first-order or plural) logic (\S\ref{sec:PotKripkeSem}). 
Finally, to illustrate the value of this general approach to strict potentialism, I apply it to a Weyl-inspired predicative set theory (\S\ref{sec:ED}), Cantor's domain principle (\S\ref{sec:DP}), and strict potentialism about Cantorian sets  (\S\ref{sec:strict-pot-ST}). 


\section{A bimodal explication of strict potentialism}\label{sec:bimodal}

I will now clarify the intended interpretation of each pair of modal operators and use this to motivate a bimodal logic that is appropriate for analyzing strict potentialism. 

\subsection{D-necessity} 

Let $\Box_D$ be the modality of being determined as true at a stage. The intuitive idea is that $w \models \Box_D \varphi$  means that $\varphi$ is \textit{made true by}---or, for that matter, \textit{grounded in}---facts about the objects available at $w$ along with generic facts that constrain objects yet to be generated (such as the law of extensionality for sets). As mentioned, it is important that D-necessity at a stage be understood as intrinsic to that stage. Thus, $w \models \Box_D\varphi$ means that $\varphi$ holds in virtue of objects and facts available at $w$, with no need for any appeal to stages beyond $w$.  Beyond that, D-necessity may, for present purposes, be treated as a black box.\footnote{See \cite{Linnebo:GenExpl} and \cite{Linnebo&Litland:GenGrounded} for some ideas about how to open this black box.} 

As we pass from one stage to another, more objects and facts become available. This justifies the 4 axiom for $\Box_D$: whatever is determined as true by material intrinsic to $w$ is also determined as true by more such material. Thus, D-necessary is monotonic:
\begin{equation} \label{eq:D-growth}
  w \leq w' \Rightarrow w \models \Box_D \varphi \Rightarrow  w' \models \Box_D \varphi  
\end{equation}
Since the T axiom too is justified, we adopt the modal logic S4 for D-modality. 

Various familiar additions to S4 are not appropriate, however. Since more truths may become D-necessary ``later'', $\neg \Box_D \varphi \to \Box_D \neg \Box_D \varphi$ often fails. This means that the 5 axiom cannot be assumed. In fact, even the weaker ``Brouwerian'' axiom B (that is, $\varphi \to \Box \Diamond \varphi$) often fails.\footnote{Free choice sequences provide simple examples of both failures. First, although it is not determined at one stage that the second entry of a sequence is 0, this may become determined at a later stage. Second, it may be open at one stage that the second entry of our free choice sequence is 0, but this possibility may later be shut down by letting this entry be 1.} Yet another axiom that cannot be assume is
\begin{equation*}
    \Diamond \Box \varphi \to  \Box \Diamond \varphi, \tag{.2}    
\end{equation*}
which is appropriate for convergent frames.\footnote{This axiom is known both as .2 and G. We will use the former name to avoid conflation with G-modality. A frame is \textit{convergent} iff any two extensions of a single world have a shared extension: $w_0 \leq w_1 \wedge w_0 \leq w_2 \to \exists w_3 (w_1 \leq w_3 \wedge w_2 \leq w_3)$. The examples from the previous note demonstrate the failure of .2 for D-modality, bearing in mind that once chosen, an entry of a choice sequence is fixed (i.e., necessary).}

\subsection{G-possibility} 

Let $\Diamond_G$ represent generative possibility. The generation in question is \textit{deterministic}, in the sense that the form of generation and the input to which it is applied  determine the output. For example, set generation applied to the objects $mm$ necessarily yields $\{mm\}$.\footnote{This contrasts with free choice sequences, mentioned in the previous notes, where new entries are freely chosen.} I will generally elide the qualification `deterministic'. 

The intuitive idea is that $w \models \Diamond_G \varphi$ means that $\varphi$ can be brought about by principles of generation and material to which these can be applied, all available at $w$. 
For example, if a number $n$ and the successor-of operation are available, the successor of $n$ can be generated. Or, if $mm$ and the set-of operation are available, the set $\{mm\}$ can be generated. 

As is customary, we take the modality of (deterministic) generation to be (at least) S4.2.\footnote{See, e.g., \cite{Linnebo:2013-PHS} and \cite{Linnebo&Shapiro:ActPotInf}.} Axiom .2 is justified because the license to generate an object is never revoked. After all, generation never ``uses up'' a generative operation or the material to which the operation is applied. Thus, if you choose not to generate some objects right away, you can always generate them later---and thus achieve convergence. 

By contrast, the axioms 5 and B are invalid for G-modality as well as for D-modality. At some stage it is the case (and thus also G-possible) that there are precisely two objects. But this G-possibility will be shut down as soon as we generate more than two objects, which cannot then be ``destroyed''. 

Some terminology is needed. 
\begin{Def}[Stability]\label{def:stb}
We say that a formula $\varphi(\mathbf{u})$ is \emph{stable} iff 
\begin{gather*}
    \Diamond \varphi(\mathbf{u}) \to \Box \varphi(\mathbf{u}) \tag{\stb-$\varphi$}
\end{gather*}
Classically, the stability of a formula is equivalent to the conjunction of what we will call its \textit{positive} and \textit{negative stability}:
\begin{gather*}
    \varphi(\mathbf{u}) \to \Box \varphi(\mathbf{u}) \tag{\stb$^+$-$\varphi$}
    \\
    \neg \varphi(\mathbf{u}) \to \Box   \neg  \varphi(\mathbf{u}) \tag{\stb$^-$-$\varphi$}
\end{gather*}
\end{Def}

Since the generation represented by G-modality does not change any ``intrinsic'' properties of objects already generated, we finally add that every atomic predication is stable with respect to G-modality.

\subsection{How the two modalities interact} 

It remains to explain how the two modalities interact. 
An obvious interaction principle is the subsumption of G-necessity under D-necessity: 
\begin{equation}\label{eq:subsump}
    \Box_D \varphi \to \Box_G \varphi \tag{\textsc{Subsump}}    
\end{equation}
Whatever is determined to hold at a stage must hold throughout all of the generative possibilities available from that stage. For example, since the law of extensionality is determined for sets, we cannot generate non-extensional sets. 

A subtler interaction principle states that certain generative possibilities cannot be lost. If at a stage the set of $mm$ can be generated, for example, it is determined that this set can be generated. But we have to be careful. We just observed that G-possibilities can be lost: although it was once possible that there are only two objects, this possibility was lost as soon as we generated a third object. 
The correct statement is that the G-possibility of a \textit{determinate} truth is itself determinate: 
\begin{equation*} \label{eq:mixed-G}\tag{\textsc{Mixed}.2}
    \Diamond_G \Box_D \varphi \to \Box_D\Diamond_G \varphi
\end{equation*} 
This captures the idea that the generative possibilities available at a stage are intrinsic to this stage and thus available at any extended stage as well.

By contrast, new generative possibilities can become available at D-later worlds. Some examples 
are provided in Section~\ref{sec:ED}. This means that we cannot assume $\neg \Diamond_G \varphi \to \Box_D \neg \Diamond_G \varphi$.


Summing up, I have articulated a bimodal logic BM, which explicates the key strict potentialist ideas of successive generation of objects and determination of truths. The resulting framework for strict potentialism can be used to develop specific views, such as strict potentialism about the natural numbers, predicative sets, or even Cantorian sets. 

\section{The amount of indeterminacy } \label{sec:amount-indet}

Let us reflect on how much determinacy there is on some specific form of strict potentialism. For which statements $\varphi$ do we have $\Box_D \varphi$? 

\subsection{Nested spheres of possible determinate truths }

Clearly, more and more statements become determinate as we generate more objects. For example, the statement that there exist at least three objects becomes determinate only after a few generative steps. Instead of asking which statements \textit{are} determinate, it is therefore more interesting to ask which ones \textit{become} determinate after some further generation. Accordingly, let us write $\textsc{Poss}_G(w)$ for the set of $\varphi$ such that $\Diamond_G \Box_D \varphi$ is true at $w$. Then (\ref{eq:mixed-G}) entails: 
\begin{equation}
    w \leq_D w' \Rightarrow \textsc{Poss}_G(w) \subseteq \textsc{Poss}_G(w')
\end{equation}
Thus, as we pass to later worlds, we cannot lose, but might gain, G-possible determinacy. 

Next, consider the property of still being a contender for determinacy as far as D-modality is concerned. Let us write $\textsc{Poss}_D(w)$ for the set of statements $\varphi$ such that $\Diamond_D \Box_D \varphi$ is true at $w$. 
The monotonicity of D-necessity, (\ref{eq:D-growth}), entails that the sets of contenders for determinacy cannot grow, but might shrink, as we pass to later worlds: 
\begin{equation}
    w \leq_D w' \Rightarrow \textsc{Poss}_D(w) \supseteq \textsc{Poss}_D(w')
\end{equation}

Why this asymmetry between G-possible determinacies, which might only grow, and D-possible determinacies, which might only shrink? 
The reason is that G-modality ``rules in'' possibilities, in a monotonic manner, since generative abilities are never lost. By contrast, D-modality ``rules out'' possibilities, thus shrinking the space of possibilities that are left open, also in a monotonic manner. 

Finally, (\ref{eq:subsump}) yields: 
\begin{equation}
     \textsc{Poss}_G(w) \subseteq \textsc{Poss}_D(w)
\end{equation}
Putting everything together, we obtain: 
\begin{equation}
    w \leq_D w' \leq_D \ldots 
 \ \Rightarrow \ \textsc{Poss}_G(w) \subseteq \textsc{Poss}_G(w') \subseteq \ldots \subseteq \textsc{Poss}_D(w') \subseteq \textsc{Poss}_D(w)
\end{equation}
In words: as we pass to larger and larger worlds, the collection of determinate truths is approximated from below (by G-modality) and from above (by D-modality).

\subsection{The question of Reverse Subsumption} 

The two approximations may or may not converge on a single sphere of determinate truths. This question of convergence is of fundamental importance.  
Suppose that Reverse Subsumption, defined as 
\begin{equation}
    \Box_G \varphi \to \Box_D \varphi \tag{\textsc{R-Subsump}}    
\end{equation}
holds at some world. At that world, the two approximations converge on a single collection of determinate truths. 
Every statement that holds throughout all the generative possibilities available at that world would also be determined there. 

Reverse Subsumption would ensure the determinacy of all ``natural'' mathematical statements, in a precise sense defined below.\footnote{This ``natural'' class consists of the statements obtained by applying the composite translation defined in Section~\ref{sec:FO-essentials}.} Since the bimodal logic is classical, every such statement either holds throughout all the available generative possibilities or not. Either way, that fact would, thanks to Reverse Subsumption, be determined.\footnote{This follows from Lemmas~\ref{Lem:G-stab} and \ref{Lem:RS}, established below.} 
Thus, every ``natural'' statement $\varphi$ would be determined one way or the other: $\Box_D \varphi \vee \Box_D\neg \varphi$. 
We will find that, when D-modality is switched off, this determinacy would license classical quantification across the relevant hierarchy of G-possible objects. 

Alternatively, suppose that Reverse Subsumption fails. Then D-modality would be strictly broader than G-modality. The two approximations of determinate truth would then fail to converge and instead leave a middle band of indeterminacy that extends into the class of ``natural'' mathematical statements. We will find that, when D-modality is switched off, this scenario licenses only intuitionistic logic for quantification across the relevant hierarchy of G-possible objects.

What view might license Reverse Subsumption? We would need it to be determined at some world $w$ that no further generative possibilities will ever arise. That would ensure that whatever holds across all the generative possibilities at $w$ is also determined at $w$---precisely as Reverse Subsumption states. My view is that Reverse Subsumption is invalid. Some plausible counterexamples are described in the three last sections.  

By contrast, certain \textit{relativized} forms of Reverse Subsumption, obtained by restricting our attention to objects that are generated in some particular way, can be justified. Consider the natural numbers, which are generated by the successor-of operation. Plausibly, it is determined that there is no other way to generated natural numbers. This thought leads predicativists (and anyone with stronger views) to accept Reverse Subsumption for quantification relativized to natural numbers: 
\begin{equation*}
    \Box_G \forall n \varphi \to \Box_D \forall n\varphi
\end{equation*}
And this, in turn, will be shown (in Section~\ref{sec:ED}) to justify classical logic for the relativized quantifiers `$\forall n$' and `$\exists n$' when D-modality is switched off. 

In short, the greater the distance between G- and D-modality on some form of strict potentialism, the more indeterminacy this view will entail, and---as will be shown---the greater the need for intuitionistic logic when D-modality is switched off. Thus, the distance between the two modalities provides a measure of the non-classicality to which a form of strict potentialism is committed.

\section{First-order unimodal mirroring: a review }\label{sec:FOL-mirroring}

I will now describe some existing results---and develop some new ones---that allows us to turn the two modalities on and off. We gradually work our way up to the cases that most interest us. We begin, in this section, by reviewing some known results about first-order unimodal systems. These results are extended in the next section to a first-order version of our bimodal system. In later sections, we press on to the desired systems involving a plural logic. 

What we may call \textit{G\"odel mirroring} allows D-modality to be switched on or off. This corresponds to a choice between the classical modal logic S4 and intuitionistic non-modal logic, respectively. What I call \textit{potentialist mirroring} allows G-modality to be switched on or off. Although doing so makes no difference concerning the validity of arguments in (non-modal) first-order logic, we will later find that a sharp difference emerges for plural logic.

\subsection{G\"odel mirroring}\label{sec:Godelmirror}

The G\"odel translation $g$ is a translation from the language of intuitionistic logic into that of classical modal logic. The translation is given by the following clauses:
\begin{align*}
     \varphi  \mapsto \Box \varphi \qquad \textrm{for $\varphi$ atomic}  &&
    \\
    \varphi \vee \psi  \mapsto \varphi^g \vee \psi^g && 
    \varphi \wedge \psi  \mapsto \varphi^g \wedge \psi^g \\
    \varphi \to \psi  \mapsto \Box (\varphi^g \to \psi^g) &&  \neg \varphi  \mapsto \Box \neg \varphi^g \\
    \exists x\, \varphi  \mapsto \exists x \, \varphi^g&& 
    \forall x\, \varphi   \mapsto \Box \forall x \, \varphi^g
\end{align*}
The translation can be extended to plural and second-order logic in an analogous way.

Let $\vdash_\textrm{S4}$ be deducibility in the logic that results from combining S4 with classical first-order logic. 
As is well known, this logic  proves the Converse Barcan Formula,
\begin{equation}
\exists x \Diamond \varphi \to \Diamond \exists x \varphi. \tag{CBF}
\end{equation}

It is easy to prove by induction on syntactic complexity that $\vdash_\textrm{S4} \varphi^g \to \Box \varphi^g$. 
Thus, in S4 we get a partial ``modal collapse'' for formulas in the image of $g$, since $\varphi^g$ and $\Box \varphi^g$ are provably equivalent. However, it is also easy to see that the modal collapse is not complete, since we do not have $\Diamond \varphi^g \to \varphi^g$ (or equivalently $\neg \varphi^g \to \Box \neg \varphi^g$). 
Using the notions of stability from Definition~\ref{def:stb}, our findings can be summarized as follows.
\begin{Lem}\label{Lem:G-stab} 
Any formula in the image of the G\"odel translation $g$ is positively stable but need not be negatively stable. 
\end{Lem}

It is well known that the G\"odel translation validates intuitionistic logic when the logic of the modal language is S4. More precisely, we have the following theorem.\footnote{See \cite{RasiowaSikorski:1953} for a proof.} 
\begin{Thm} \label{thm:mirroring-G}
Let $\vdash_\textrm{int}$ be intuitionistic deducibility in the given language, but if the language is plural or higher-order, we remove any comprehension axioms. Let $\vdash_\textrm{S4}$ be the corresponding deducibility relation in S4 and classical logic. Then we have:
\[
\varphi_1, \ldots, \varphi_n \vdash_\textrm{int} \psi \textrm{\quad iff \quad} \varphi^g_1, \ldots, \varphi^g_n \vdash_\textrm{S4} \psi^g.
\]
\end{Thm}
We call theorems of this sort \emph{mirroring theorems}. The left-to-right direction states that we have an \textit{embedding} of intuitionistic logic in classical S4, while the reverse direction says that the embedding is \textit{faithful}.

\subsection{First-order potentialist mirroring}\label{sec:potmirror}

In mathematics we usually quantify over all objects regardless of the stage at which they are generated. When we say `for all', we mean \emph{for all that will ever be generated}, and when we say `there is', we mean that \emph{we can generate an object such that}. Such unrestricted quantifiers correspond to \emph{modalized} quantifiers of the modal framework: $\Box \forall x$ and $\Diamond \exists x$.

\begin{Def}
The \emph{potentialist translation}  $\varphi \mapsto \varphi^\Diamond$ replaces each ordinary quantifier with the corresponding modalized quantifier.
  A formula is \emph{fully modalized} iff it is in the image of this translation. 
\end{Def}

\begin{Thm}[Classical potentialist mirroring]\label{thm:class-pot-FOL-mirror}
Let $\vdash^\Diamond$ be provability by $\vdash$, S4.2, and axioms stating that every atomic predicate is stable, but with no higher-order comprehension. Then we have: 
\[
\varphi_1, \ldots, \varphi_n \vdash \psi \textrm{\quad iff \quad} \varphi^\Diamond_1, \ldots, \varphi^\Diamond_n \vdash^\Diamond \psi^\Diamond.
\]
\end{Thm}
See \cite[Thm.~5.4]{Linnebo:2013-PHS} for a proof, which proceeds by an induction on proofs.

An essential stepping stone towards this theorem is the following lemma.
\begin{Lem}[Stability]\label{Lem:ModalCollapse}
Let $\varphi$ be a fully modalized formula of a modal language $\mathcal{L}^\Diamond$. Then S4.2 and the stability axioms for $\L^\Diamond$ prove that $\varphi$ is stable. Hence we also have that $\Diamond \varphi$, $\varphi$, and $\Box \varphi$ are equivalent.
\end{Lem}

What happens to Theorem~\ref{thm:class-pot-FOL-mirror} when we let the logic of the modal language be intuitionistic rather than classical?\footnote{For a description of intuitionistic S4.2, see \cite{Linnebo&Shapiro:ActPotInf}, which in turn draws on \cite{SimpsonAlex-int-modal-logic}.} The question is answered by a theorem due to \cite{Linnebo&Shapiro:ActPotInf}. We here state the result incorporating an improvement from \cite{Crosilla-Linnebo:2024-Weyl}.

\begin{Thm}[Intuitionistic potentialist mirroring]
\label{thm:int-pot-FOL-mirror}
Let $\vdash_\textrm{int}$ be intuitionistic deducibility in the given language, but as before, if the language is plural or higher-order, we remove any comprehension axioms. Let $\vdash^\Diamond_\textrm{int}$ be the corresponding deducibility relation in intuitionistic S4.2, plus the stability axioms in the form $\Diamond \varphi \to \Box \varphi$. 
Then, for any non-modal formulas $\varphi_1, \ldots, \varphi_n, \psi$, we have: 
\[
\varphi_1, \ldots, \varphi_n \vdash_\textrm{int} \psi \textrm{\quad iff \quad} \varphi^\Diamond_1, \ldots, \varphi^\Diamond_n \vdash^\Diamond_\textrm{int} \psi^\Diamond.
\]
\end{Thm}

\section{First-order bimodal mirroring}\label{sec:bi-mirroring-FOL}

We define the first-order bimodal logic BM-FOL  as the result of adding the bimodal logic described in Section~\ref{sec:bimodal} to classical first-order logic.

I first present the essentials, then the full story. The latter is intended for \textit{aficionados} only; others may skim or even skip, without compromising their ability to understand the remainder of the article. The same goes for the analogous discussion Section~\ref{sec:bi-mirroring-PL}.

\subsection{The essentials}\label{sec:FO-essentials}

Each of the two modalities can be switched on or off. Switched on, a modality is left as an explicit part of the syntax of a modal logic. Switched off, a modality is absorbed into 
the logical operators, in accordance with the two translations just described.
\footnote{With two modalities switched on, each translation need a simple extension whose precise definition can be found in Sect.~\ref{sec:FO-fullstory}.} Depending on whether D-modality is switched on or off, the associated logic will be either classical or intuitionistic.

The four resulting languages and the translations between them are depicted by the following diagram 
\vspace{2mm}
\[ 
\xymatrix{
 \L \ar@{->}[r]^\Diamond \ar@{->}[d]_g& \L^G\ar@{->}[d]^g \\ 
\L^D \ar@{->}[r]^\Diamond & \L^\textrm{BM} 
}
\]
\vspace{0mm}

\noindent where $\L$ is the language of intuitionistic first-order logic, I-FOL, and the other three languages add one or both modalities. 
By switching both modalities off, we obtain two composite translations from $\L$ to $\L^\textrm{BM}$. 

Some straighforward work in BM establishes: 
\begin{Lem}\label{lem:composite-trans}
The two composite translations commute, \textit{modulo} consequence in BM-FOL, being both equivalent to the simplified translation $\varphi \mapsto \varphi^\ast$ defined as follows: 
\begin{align*}
     \varphi  \mapsto \Box_D \varphi \qquad \textrm{for $\varphi$ atomic}  &&
    \\
    \varphi \vee \psi  \mapsto \varphi^\ast \vee \psi^\ast && 
    \varphi \wedge \psi  \mapsto \varphi^\ast \wedge \psi^\ast \\
    \varphi \to \psi  \mapsto \Box_D (\varphi^\ast \to \psi^\ast) &&  \neg \varphi  \mapsto \Box_D \neg \varphi^\ast \\
    \exists x\, \varphi  \mapsto \Diamond_G \exists x \, \varphi^\ast&& 
    \forall x\, \varphi   \mapsto \Box_D \forall x \, \varphi^\ast 
\end{align*}
Note that $\ast$ is just like the ordinary G\"odel translation $g$, with the single exception that the existential quantifier is translated in a potentialist-friendly manner as $\Diamond_G \exists$. 
\end{Lem}

\begin{Thm}[First-order bimodal mirroring]\label{thm:BM-FOL}
\begin{enumerate}[(a)]
    \item The translation $\ast$ yields a faithful embedding of I-FOL in BM-FOL.
    \item If we add Reverse Subsumption to BM-FOL, we obtain a faithful embedding of classical FOL. 
\end{enumerate}
\end{Thm}
\textit{Proof}. 
(a) and (b) follow from Theorem~\ref{thm:mirroring-G} combined with Lemmas~\ref{lem:pot-BM-FOL} and \ref{lem:unfaithful} below, respectively, as spelled out more fully below. $\dashv$

\subsection{An application to strict potentialist arithmetic}

To illustrate the bimodal mirroring of Theorem~\ref{thm:BM-FOL}(a), let me disgress to apply the result to strict potentialist arithmetic. Since the application is straightforward and holds no surprises, I will be very brief. 

We begin by formulating Heyting Arithmetic in a relational signature. Thus, instead of a successor function, we have the axiom $\forall m \exists n \textsc{Succ}(m,n)$; likewise for all the other axioms. 

Next, we apply the translation $\ast$.
We observe that each of the resulting statements is justified from a strict potentialist point of view. For example, the availability of the successor-of operation makes it determinate that for any given number we can generate its successor: 
\begin{equation}
    \Box_D \forall m \Diamond_G \exists n \textsc{Succ}(m,n)
\end{equation}

Finally, we invoke 
Theorem~\ref{thm:BM-FOL}(a) to observe that this bimodal theory corresponds to the version of HA with which we started. 

The last three sections develop applications of my bimodal analysis of strict potentialism that are more interesting, conceptually as well as mathematically.

\subsection{The full story} \label{sec:FO-fullstory}

As observed, there are four options, depending on which of the two modalities is switched on or off. The following table displays the corresponding theories that we will now consider: 
\begin{center}
\begin{tabular}{c|cc}
     & $\Diamond_G$ off & $\Diamond_G$ on 
     \\ \hline
  $\Box_D$ off & I-FOL &S4.2-I-FOL\\
  $\Box_D$ on & S4-FOL & BM-FOL\\
\end{tabular}
\end{center}
\vspace{2mm}
As stated, I-FOL and BM-FOL are intuitionistic and bimodal first-order logic, respectively. S4-FOL is classical S4 added to FOL, and S4.2-I-FOL is as described in Section~\ref{sec:potmirror}. 
Both of the systems on the right-hand side include the G-stability of atomic predications, formulated as $\Diamond_G \varphi \to \Box_G \varphi$. 

\begin{Lem}\label{lem:pot-BM-FOL}
    The potentialist translation provides a faithful interpretation of each system on the left of the above table in the corresponding system on the right.\footnote{By `interpretation', I mean an embedding based on a  translation that commutes with the logical connectives.} 
\end{Lem} 

\noindent \textit{Proof}. The upper row is the intuitionistic potentialist mirroring of Theorem~\ref{thm:int-pot-FOL-mirror}. 

Next, consider the lower row. We extend the potentialist translation $\varphi \mapsto \varphi^{\Diamond}$ to the language of S4-FOL by translating $\Box \varphi$ as $\Box_D (\varphi^{\Diamond_G})$ and otherwise proceeding as on the ordinary potentialist translation, only using G-modality. (For example, $\exists x \, \varphi$ is translated as $\Diamond_G \exists x \,\varphi^{\Diamond_G}$.) We build on the classical potentialist mirroring of Theorem~\ref{thm:class-pot-FOL-mirror}. 
To show that we have an interpretation, we observe that each axiom of S4 is mapped to a corresponding axiom of BM, and that Necessitation corresponds to D-necessitation. We prove that the interpretation is faithful by dropping the G-modality, as in the proof of Theorem~\ref{thm:class-pot-FOL-mirror}. $\dashv$

\vspace{4mm}

We wish to compare the pair of systems in each column of the above table, using the G\"odel translation. We already know from Theorem~\ref{thm:mirroring-G} that the G\"odel translation provides a faithful embedding of I-FOL in S4-FOL. This handles the left column. 

Next, consider the right-hand column of the table. We extend the G\"odel translation to the language of S4.2-I-FOL. The connectives and quantifiers are translated as in the ordinary G\"odel translation, just using D-modality. The modal operators are translated as follows: 
$\Diamond \varphi \mapsto	\Diamond_G \varphi^g$ and 
$\Box\varphi \mapsto \Box_D \varphi^g$. 
\begin{Lem}\label{lem:unfaithful}
    This extended G\"odel translation provides an  embedding of S4.2-I-FOL in 
    BM-FOL, which is not faithful.  
\end{Lem} 
\textit{Proof}. We build on 
Theorem~\ref{thm:mirroring-G}, which shows that I-FOL is embedded in BM-FOL. Next, it is routine to show that each modal axiom of intuitionistic S4.2 translates as a theorem of BM. For example, the two parts of Axiom 4, $\Box \varphi \to \Box \Box \varphi$ and $\Diamond\Diamond \varphi \to \Diamond \varphi$ translate as $\Box_D \varphi \to \Box_D \Box_D \varphi$ and $\Diamond_G\Diamond_G \varphi \to \Diamond_G \varphi$. And Axiom G translates as the D-necessity of (\ref{eq:mixed-G}). The rule of Necessitation of S4.2 translates as G-Necessiation in BM. 

To show that the embedding is not faithful, consider $\varphi \to \Box \varphi$, which is not a theorem of S4.2-I-FOL. But the (extended) G\"odel translation of this formula is a theorem of BM-FOL, which proves $\varphi^g \to \Box_D \varphi^g$. The proof of this latter claim goes by induction on syntactic complexity, as in the analogous result, Lemma~\ref{Lem:G-stab}, for the ordinary G\"odel translation.  $\dashv$

\begin{Lem}\label{Lem:RS}
    When Reverse Subsumption is added, the translation $\ast$ is equivalent in BM to the potentialist translation $\Diamond$.
\end{Lem}
\emph{Proof}. BM plus Reverse Subsumption has two provably equivalent modalities, both obeying S4.2. In the context of this modal system, Lemma~\ref{Lem:ModalCollapse} entails that the  translation $\ast$ is equivalent to the potentialist translation. $\dashv$

\vspace{4mm}
It remains to spell out the compressed parts of the proof of Theorem~\ref{thm:BM-FOL}. For (a), we regard the composite translation as the G\"odel translation followed by the potentialist translation. The result is then immediate from Theorem~\ref{thm:mirroring-G} and Lemma~\ref{lem:pot-BM-FOL}, since the composition of two faithful embeddings yields another such. (By Lemma~\ref{lem:unfaithful}, it is essential for the proof that the translations be composed in the mentioned order.) Claim (b) follows from Lemma~\ref{Lem:RS} and the classical potentialist mirroring of Theorem~\ref{thm:class-pot-FOL-mirror}.


\section{Potentialist Kripke semantics 
}\label{sec:PotKripkeSem}

Ordinary Kripke semantics for intuitionistic logic is deeply inhospitable to potentialism. Consider the claim that every number has a successor, $\forall x \exists y \, \textsc{Succ}(x,y)$. Appropriated interpreted, this is something that potentialists too accept. With ordinary Kripke semantics, however, a world satisfies this claim only if its domain contains infinitely many numbers, if any. And this is precisely what potentialists deny. I will now pause to observe how the translations from I-FOL to BM-FOL that we have just described point the way to a potentialist-friendly variant of Kripke semantics for intuitionistic first-order logic.\footnote{This variant is an alternative to the unimodal potentialist-friendly Beth-Kripke semantics of \cite{Brauer-Linnebo-Shapiro:DivPot}. This unimodal semantics uses a new modal primitive of inevitability.} 

We work with the simplified translation $\ast$ of Lemma~\ref{lem:composite-trans}. The key idea is to absorb this translation into the semantic clauses for the logical operators by defining a relation $\Vdash$ between worlds and formulas. For example, we let $w \Vdash \exists x \, \varphi(x)$ iff there is a G-extension $w'$ of $w$ with an object $a$ such that $w' \Vdash \varphi(a)$, and $w \Vdash \forall x \, \varphi(x)$ iff for every D-extension $w'$ of $w$ and every object $a$ at $w'$ we have $w' \Vdash \varphi(a)$. \textit{Mutatis mutandis} for the other cases. Crucially, the semantics is designed to ensure the following Link:
\begin{equation}
    w \Vdash \varphi \textrm{ \quad iff \quad } w \models \varphi^\ast    \tag{Link}
\end{equation}
where $\models$ is satisfaction in the classical bimodal semantics. 

Furthermore, we consider bimodal frames that satisfy all the requirements articulated in Section~\ref{sec:bimodal}. To review: the accessibility relation $\leq_D$ is reflexive and transitive; $\leq_G$ is additionally convergent; moreover, $\leq_G$ is contained in $\leq_D$; and finally, we have the mixed convergence property that whenever $w_0$ has a G-extension $w_1$ and a D-extension $w_2$, then there is $w_3$ that D-extends $w_1$ and G-extends $w_2$: 

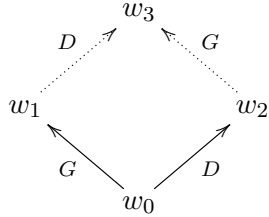
\[
\xymatrix{
 & w_3 &  \\
w_1 \ar@{.>}[ur]^D &  & w_2 \ar@{.>}[ul]_G \\
& w_0 \ar@{->}[ul]^G \ar@{->}[ur]_D & }
\]

Let \textit{potentialist Kripke semantics} be based on these bimodal frames and the relation $\Vdash$. The value of this semantics is underwritten by our next result.  

\begin{Thm}
Intuitionistic first-order logic is sound and complete with respect to potentialist Kripke semantics. 
\end{Thm}
\emph{Proof}. Clearly, any bimodal frame complying with the mentioned requirements satisfies the logic BM-FOL. Thus, by Theorem~\ref{thm:BM-FOL}(a) 
and (Link), we obtain the soundness claim. For the completeness, we use the fact that I-FOL is complete with respect to ordinary Kripke semantics, and that every ordinary Kripke model can also be regarded as a bimodal Kripke model where $\leq_G$ has shrunk to just to the identity relation. $\dashv$

\vspace{4mm}
\noindent We will see in Section~\ref{sec:PL-essentials} how this semantics can be extended to intuitionistic plural logic.

\section{Classical plural mirroring: a review}\label{sec:mono-mirroring-PL}

Our aim is to establish plural \textit{bimodal} mirroring. As a final stepping stone, this section introduces some systems of classical plural logic and states a known result about \textit{unimodal} plural mirroring.

\subsection{Minimal plural logic} \label{sec:MPL}


We start with the usual axioms and inference rules for the plural quantifiers, excluding comprehension axioms. Then we add an extensionality principle for pluralities, to the effect that any two coextensive pluralities are indiscernible (with respect to non-intensional contexts). Lastly, since pluralities are arbitrary, combinatorial collections, we add the following principle of Plural Choice:\footnote{All axiom schemes are understood to include instances from any language in which they
figure, including the modal language to be adopted shortly.} 
\begin{quote}
    Assume $(\forall x \prec xx)( \exists y \prec yy) \psi(x,y)$ and $\psi(x,y) \wedge \psi(x',y) \to x=x'$. Then there are $zz$ such that $(\forall x \prec xx )(\exists! y \prec zz ) \psi(x,y)$.
\end{quote} 
We call the resulting system \textit{Minimal Plural Logic}.\footnote{One might want to require that every plurality be non-empty, in which case we would need the assumption that there be at least one $\varphi$. Nothing of mathematical interest hangs on our decision to waive this requirement.}

\subsection{Plural comprehension}

Stronger plural logics are obtained by adding various plural comprehension schemes. The strongest of these systems is \textit{Traditional Plural Logic}, which adds an unrestricted plural comprehension scheme:
\begin{equation}
    \exists yy \forall x (x \prec yy \leftrightarrow \varphi(x)) \tag{P-Comp}
\end{equation}

We will also be interested in some more restrictive systems. Let \textit{Basic Plural Logic} (BPL) be the plural logic that adds to Minimal Plural Logic the following four axioms or axioms schemes. 
First, there is an empty plurality. Second, we have an axiom of adjunction: 
\begin{equation}\label{eq:P-Adj}
    \forall x \forall xx \exists yy \forall y (y \prec yy \leftrightarrow y \prec xx \wedge y = x) \tag{P-Adj}
\end{equation}
Next, we adopt an axiom of plural pairwise union:
\begin{equation} \label{eq:P-Union}
    \exists zz \forall x (x \prec zz \leftrightarrow x \prec xx \vee x \prec yy) \tag{P-Union}
\end{equation}
Finally, we add an axiom scheme of plural separation: 
\begin{equation}\label{eq:P-Sep}
    \exists yy \forall x (x \prec yy \leftrightarrow x \prec xx \wedge \varphi(x)) \tag{P-Sep}
\end{equation}
I will eventually, in Section~\ref{sec:DP},  recommend a stronger, though still non-traditional, plural logic.

\subsection{The modal rigidity of pluralities}\label{sec:modallogicplurals}

We turn now to logical principles concerning the interaction of modals and plurals. Our guiding idea is that every plurality is modally rigid. Consider some objects. Wherever these objects exist, they comprise the very same objects. After all, all we have to go on when tracking these objects across different possibilities are precisely \textit{these objects}. 

First, we lay down that both plural membership and plural identity are stable: 
\begin{align*}\label{eq:stb-prec}  \tag{\textsc{Stb}-$\prec$}        x \prec yy \rightarrow\Box x \prec yy && \lnot x \prec yy \rightarrow\Box \lnot x \prec yy 
\end{align*}
and likewise for identity. 
Second, we adopt the following inextendability principles for $\prec$ and $\preccurlyeq$, to the effect that a plurality cannot gain members or subpluralities as we move to more populous ``worlds'':\footnote{These principles are Barcan formulas for the restricted quantifier $\forall x \prec yy)$ and its plural analogue. See \cite[pp.~211-12]{Linnebo:2013-PHS} for further explanation. }    
\begin{gather*}
    \forall x (x \prec yy \to \Box \theta) \to \Box \forall x (x \prec yy \to \theta) \label{eq:Inext-prec} \tag{\textsc{InExt}-$\prec$}\\
    \forall xx (xx \preccurlyeq yy \to \Box \theta) \to \Box \forall xx (xx \preccurlyeq yy \to \theta)\tag{\textsc{InExt}-$\preccurlyeq$}\label{eq:Inext-preccurlyeq}
\end{gather*}
I will refer to the axioms described in this subsection as \textit{the modal rigidity axioms}. 

\subsection{Potentialist mirroring for classical plural logic}\label{sec:pluralpotmirror}

Let \textit{Modal Traditional Plural Logic (M-TPL)} be the system obtained from Traditional Plural Logic by adding S4.2, the stability axioms for atomic predicates, and principles of modal rigidity just described. We can now state the promised classical plural mirroring theorem.\footnote{See \cite{Linnebo:PotDemodal} for a proof. The theorem was first presented in a talk I gave to the Oslo Logic Seminar on 19 November 2015. In fact, by tweaking the two systems, this theorem can be extended to a result about definitional equivalence. I cannot see an analogous extension for the bimodal mirroring theorems stated below. }

\begin{Thm}[Classical Plural Mirroring]\label{thm:plural-mirroring}
    The potentialist translation $\varphi \mapsto \varphi^\Diamond$ provides a faithful interpretation of BPL in M-TPL. That is, we have: 
    \[
\varphi_1, \ldots, \varphi_n \vdash^{\text{BPL}} \psi \text{\quad iff \quad} \varphi^\Diamond_1, \ldots, \varphi^\Diamond_n \vdash^\text{M-TPL} \psi^\Diamond
\]
\end{Thm}
The upshot is that potentialists who rely on the usual classical modal logic of plurals is thereby entitled to basic plural logic. 
Analogous results hold for pairs of systems that strengthen, in matching ways, both the modal and the non-modal sides.

\section{Plural bimodal mirroring}\label{sec:bi-mirroring-PL}

After the unimodal detour of the previous section, let us return to the bimodal setting. Where Section~\ref{sec:bi-mirroring-FOL} studied \textit{first-order} bimodal mirroring, we are now ready to tackle the case of \textit{plural} logic. 

As before, each modality---$\Box_D$ for determination and $\Diamond_G$ for generative potential---can be left on or switched off. This yields four systems, which I will first describe 
and then compare by means of mirroring theorems. 
Our starting point are the first-order systems from Section~\ref{sec:bi-mirroring-FOL}, with the two important ones in boldface, whereas the other two are for \textit{aficionados} only:
\vspace{2mm}
\begin{center}
\begin{tabular}{c|cc}
     & $\Diamond_G$ off & $\Diamond_G$ on 
     \\ \hline
  $\Box_D$ off & \textbf{I-FOL} &S4.2-I-FOL\\
  $\Box_D$ on & S4-FOL & \textbf{BM-FOL}\\
\end{tabular}
\end{center}
\vspace{2mm}
Next, we add the Minimal Plural Logic of Section~\ref{sec:MPL} to each system. 
What remains are axioms concerning the modal rigidity of pluralities and plural comprehension. Let us consider these in turn.

\subsection{The modal rigidity of pluralities, intuitionistically}\label{sec:rigid-pl}

We begin with ``the bottom floor'' of the above table, where classical logic is available. 
Here we simply adopt the modal rigidity axioms 
from Section~\ref{sec:modallogicplurals} with respect to whichever (one or two) modalities are available. 

The two systems on the ``top floor'' of our table, where the logic is intuitionistic, are less straightforward but more interesting. In these systems, D-modality has been switched off in favor of intuitionistic logic. The rigidity of pluralities with respect to D-modality turns out to have consequences concerning the corresponding intuitionistic plural logic. 
First, the decidability of plural membership $\prec$
\begin{equation}\label{eq:Dec-prec}
    x \prec yy \vee \neg x \prec yy \tag{\textsc{Dec}-$\prec$}
\end{equation}
corresponds to the D-stability of  $\prec$.
For we easily show: 
\begin{Lem}\label{Lem:godel-prec}
    The original (or extended) G\"odel translation of (\ref{eq:Dec-prec}) is provable in S4 (or BM) from the D-stability of $\prec$. 
\end{Lem}

Next, consider the following \textit{omniscience principle} for membership in a given plurality $aa$: 
\begin{equation}\label{eq:P-Omni}
    (\forall x \prec aa)(\varphi(x) \vee \neg \varphi(x)) \to (\exists x \prec aa)\varphi(x) \vee \neg (\exists x \prec aa) \varphi(x) \tag{\textsc{Omni}$\prec$}
\end{equation}
This principle states that quantification restricted to $aa$ behaves classically. More precisely, provided that $\varphi(x)$ behaves classically on any member of $aa$, the result of prefacing $\varphi(x)$ with the restricted quantifier $\exists x \prec aa$ yields a formula that in turn behaves classically. This omniscience principle turns out to correspond to the D-inextendability of $\prec$: 
\begin{Lem} \label{Lem:godel-omni}
    The original (or extended) G\"odel translation of (\ref{eq:P-Omni}) is provable in S4 (or BM) from the D-inextendability of~$\prec$. 
\end{Lem}
\textit{Proof}. The antecedent of (\ref{eq:P-Omni}) translates as $(\forall x \prec xx) (\varphi^\ast(x) \vee \neg \varphi^\ast(x))$. Thus, classical modal logic yields: 
\begin{equation*}
    (\exists x \prec aa) \varphi^\ast(x) \vee \neg (\exists x \prec aa) \varphi^\ast(x)
\end{equation*}
By the D-inextendability of $\prec$, can put $\Box_D$ in front of the second disjunct. (And in BM, we can put $\Diamond_G$ in front of the first). This yields the translation of the consequent of (\ref{eq:P-Omni}). $\dashv$

\vspace{4mm}
Finally, we have the analogous definitions and results concerning $\preccurlyeq$ instead of $\prec$. 

Let me summarize the principles of modal rigidity of pluralities to be adopted on the top floor of our table. In the system where both modalities are switched off, we adopt (\ref{eq:Dec-prec}), (\ref{eq:P-Omni}) and its analogue (\textsc{Omni}$\preccurlyeq$) for $\preccurlyeq$. (For \textit{aficionados}: in the system where D-modality is switched off but G-modality is left on, we adopt the mentioned intuitionistic principles as well as the rigidity axioms concerning G-modality.)

\subsection{Plural comprehension, intuitionistically}\label{sec:int-P-comp}

We begin with the bimodal system, where both modalities are switched on. Here we simply adopt unrestricted plural comprehension (and thus Traditional Plural Logic). 

Next, we take a step to the left, by switching off G-modality. This requires the weaker comprehension principles of Basic Plural Logic---namely, Singleton, Pairwise Union, and Plural Separation---just as in the case of unimodal plural mirroring. 


Then, we take a step up to consider the important system where both modalities are switched off. Since G-modality is switched off, we again require some variant of Basic Plural Logic. But since plural membership is decidable, we additionally need to restrict Plural Separation to require a decidable condition: 
\begin{equation}\label{eq:Dec-P-Sep}
    (\forall x \prec xx) (\varphi(x) \vee \neg \varphi(x)) \to \exists yy \forall x (x \prec yy \leftrightarrow x \prec xx \wedge \varphi(x)) \tag{\textsc{Dec}-P-Sep}
\end{equation}
Let \textit{decidable BPL-comprehension} be plural comprehension as in Basic Plural Logic though with this restriction. 

(Finally, for \textit{aficionados}, there is the upper right-hand side, where D-modality is off but G-modality is on. Suppose we had full comprehension here. Then we would have a universal plurality: $\exists yy \forall x \, x \prec yy$. By making the D-modality explicit, this would yield a plurality that is \textit{D-necessarily} universal, which is too strong. Thus, we need some form of Basic Plural Logic here as well. Since plural membership is required to be decidable, Plural Separation must again be restricted to require a decidable condition.) 

\subsection{The essentials}\label{sec:PL-essentials}

Let me summarize the two important systems. First, I-BPL is obtained from I-FOL by adding 
\begin{enumerate}[(i)]
    \item Minimal Plural Logic; 
    \item the rigidity principles (\ref{eq:Dec-prec}), (\ref{eq:P-Omni}), and (\textsc{Omni}$\preccurlyeq$);
    \item decidable BPL-comprehension. 
\end{enumerate}
Second, BM-TPL is obtained from BM-FOL by adding 
\begin{enumerate}[(i)]
    \item Minimal Plural Logic; 
    \item the rigidity principles (\ref{eq:stb-prec}), (\ref{eq:Inext-prec}), and (\ref{eq:Inext-preccurlyeq}) for both modalities; 
    \item unrestricted plural comprehension. 
\end{enumerate}

By combining the G\"odel translation and the potentialist translation, there are two composite translations from the language of I-BPL to that of BM-TPL. As before, these are equivalent, \textit{modulo} BM, to the simplified translation $\ast$ of Lemma~\ref{lem:composite-trans}. 
We also obtain (almost) an analogue of Theorem~\ref{thm:BM-FOL} concerning the corresponding first-order system: 
\begin{Thm}[Bimodal plural mirroring]\label{thm:BM-PL}
\begin{enumerate}[(a)]
    \item The translation $\ast$ yields an embedding of I-BPL in BM-TPL. 
    \item When we add Reverse Subsumption to BM-TPL, we obtain a faithful embedding of classical BPL. 
\end{enumerate}
\end{Thm}
It is an open question whether the embedding in (a) is faithful. 

As for the proof, (b) is handled as in Theorem~\ref{thm:BM-FOL}, whereas (a) is a consequence of the more general Theorem~\ref{thm:fullstory-plural} established below. In contrast to Theorem~\ref{thm:BM-FOL}, it is unknown whether the embedding in (a) is faithful. 

Finally, it is straightforward to extend the potentialist Kripke semantics from Section~\ref{sec:PotKripkeSem} to plural logic. As in the first-order case, it is easy to see that I-BPL is sound with respect to this semantics. However, I-BPL is not complete with respect to the semantics because of its standardness assumption, to the effect that the plural variables at a world range over \textit{all} pluralities of objects from the domain of that worlds.\footnote{\cite{Florio&Linnebo-HO-Henkin} shows there is a plurality-based form of Henkin semantics that lifts this assumption. By extending this technique to bimodal Kripke models, it should be possible to prove a completeness result concerning I-BPL. }

\subsection{The full story}

For \textit{aficionados}, the full story involves four systems, corresponding to each of the two modalities being switched on or off. Another table displays matching groups of axioms that are adopted in each of these systems---on top of their modal first-order bases and Minimal Plural Logic: 
\vspace{2mm}
\begin{center}
\begin{tabular}{c|c c c c c }
   & G-\textsc{Stb} & D-\textsc{Stb}& G-\textsc{Inext}& D-\textsc{Inext} & Comprehension \\ \hline
  BM-TPL  & \textsc{Stb}$^G$$\prec$ & \textsc{Stb}$^D$$\prec$ & \textsc{Inext}$^G$$\prec$/$\preccurlyeq$ & \textsc{Inext}$^D$$\prec$/$\preccurlyeq$  & TPL \\
  S4.2-I-BPL & \textsc{Stb}$^G$$\prec$ & \textsc{Dec}$\prec$ & \textsc{Inext}$^G$$\prec$/$\preccurlyeq$ & \textsc{Omni}$\prec$/$\preccurlyeq$  & I-BPL \\ 
  S4-BPL  & -- & \textsc{Stb}$^D$$\prec$ & -- & \textsc{Inext}$^D$$\prec$/$\preccurlyeq$ & BPL  \\
  I-BPL  & -- & \textsc{Dec}$\prec$ & -- & \textsc{Omni}$\prec$/$\preccurlyeq$  & I-BPL \\ 
\end{tabular}
\end{center}
\vspace{2mm}
Here is how to read the table. Each row describes one of the systems. The column to the left of the vertical line gives the name of each  system. The columns to the right of the line represent families of related axioms. Each entry here specifies which axiom (if any) from the relevant family is present in the relevant system. `\textsc{Inext}$^G$$\prec$/$\preccurlyeq$' is used to indicate G-inextendability of both $\prec$ and $\preccurlyeq$, and likewise for the other axioms in the relevant two columns. As concerns Comprehension, we use the name of each theory as shorthand for the comprehension axioms associated with that theory (which, strictly speaking, contains other axioms as well).

To state our mirroring results, we recall the table displaying the four systems to be related: 

\vspace{2mm}
\begin{center}
\begin{tabular}{c|cc}
     & $\Diamond_G$ off & $\Diamond_G$ on 
     \\ \hline
  $\Box_D$ off & I-BPL &S4.2-I-BPL\\
  $\Box_D$ on & S4-BPL & BM-TPL\\
\end{tabular}
\end{center}
\vspace{2mm}
The results are (almost) analogous to the ones from Section~\ref{sec:FO-fullstory} on the first-order case. 
\begin{Thm}\label{thm:fullstory-plural}
\begin{enumerate}[(a)]
    \item The (original or extended) potentialist translation provides a faithful interpretation of each system on the left in the corresponding system on the right. 
    \item The (original) G\"odel translation provides an embedding of the top left system in the bottom left one. 
    \item The (extended) G\"odel translation provides a unfaithful embedding of the top right system in the bottom right one. 
\end{enumerate}    
\end{Thm}
The proof, which is relegated to an appendix, is mostly similar in character to that of the analogous first-order claims in Section~\ref{sec:FO-fullstory}. The single major difference concerns the faithfulness of the horizontal interpretations via the potentialist translation. Here we need an essentially new idea, namely, a reverse translation developed to prove Theorem~\ref{thm:plural-mirroring}.


\subsection{Turning to applications}

This completes my general analysis of strict potentialism. I first articulated the bimodal logic BM, with a modality for each of the strict potentialist's two key ideas of object generation and truth determination. In search of more user-friendly explications, I then showed how each or both of these modalities can be switched off in favor of a less classical logic. 

To illustrate the value of this general analysis and the associated results, I will apply them, in the next three sections, 
to provide interesting analyses of some specific views in the foundations of mathematics.

\section{Predicative set theory}\label{sec:ED}

The first example is an approach to predicative set theory inspired by the great, philosophically oriented mathematician Hermann Weyl. The approach is based on the simple idea that all and only properties that are definite, in a certain sense, are eligible to define sets.

\subsection{Extensional definiteness}

Weyl characterizes a property as \textit{extensionally definite} when it has been properly ``demarcated''. The intended cash value of this demarcation is that quantification restricted to the property behaves classically.\footnote{To be precise, Weyl also requires that instantiation of the property behave classically. It would be straightforward to add this requirement to our subsequent analyses---at the cost of considerable extra complexity. Anyway, Theorem~\ref{thm:Int-ED}(f) shows that the additional requirement follows from an assumption that is very natural in our setting.} Extensional definiteness is thus related to, but different from, a property's defining a plurality. Since every plurality is completed, 
quantification restricted to the relevant objects behaves classically---as argued in Section~\ref{sec:rigid-pl}. Every plurality is therefore extensionally definite.\footnote{Strictly speaking, the claim concerns the property of being a member of the plurality. 
} 

However, a property can be extensionally definite without being completed (and thus defining a plurality). 
The property of being a natural number provides an important example. Weyl maintains that this property is extensionally definite but insists that ``inexhaustibility is essential to the infinite'' \cite[p.~23]{Weyl:1918}. This inexhaustibility means that the natural numbers cannot all be available, for us to ``go through one by one, like a police officer in his files'' \cite[p.~87]{Weyl:Grundlagenkrise}.

Extensional definiteness admits of a pleasing analysis in our bimodal framework. The basic idea is that a property is extensionally definite (or \textit{ED} for short) at a stage $s$ just in case $s$ contains the resources to generate all instances of the property---and determinately so. Consider the generation of natural numbers. There is a single way to generate these, namely, by means of the \textit{successor-of} operation; it is determined that no other way will ever become available. Thus, whatever is true of all G-possible natural numbers is also true of all D-possible numbers. We can capture this as follows:  
\begin{equation} \label{eq:ED-N}
    \Box_G (\forall x: \mathbb{N}x) \Box_D \theta \to \Box_D (\forall x: \mathbb{N}x)  \theta \tag{ED-$\mathbb{N}$}
\end{equation}
In other words: as concerns natural numbers, G-modality and D-modality coincide. 

More generally, we say that a property  $\lambda x.\varphi(x)$ is ED iff, for any $\theta$, we have: 
\begin{equation} \label{eq:D-inext-rel-G}
    \Box_G (\forall x: \varphi(x)) \Box_D \theta \to \Box_D (\forall x: \varphi(x))  \theta \tag{ED}
\end{equation}
That is, $\varphi(x)$ is ED just in case the $\varphi$s we can generate are determinately all the $\varphi$s---thus ensuring that, as concerns $\varphi$s, the two modalities coincide.
Clearly, this analysis of ED-ness is only available in the context of strict potentialism, where we have two modalities to play with, not just one, as in non-strict (or ``liberal'') potentialism. 

Let us examine what ED-ness amounts to when the modalities are switched off. 
First, we switch off G-modality to obtain: 
\begin{equation} \label{eq:D-completion-G-off}
    (\forall x: \varphi(x)) \Box_D \theta \to \Box_D (\forall x: \varphi(x))  \theta  \notag 
\end{equation}
This states that the property $\lambda x. \varphi(x)$ is D-inextendable. 
Then, we switch off D-modality as well. By Lemma~\ref{Lem:godel-omni}, the D-inextendability of $\varphi(x)$ corresponds to omniscience for quantification restricted to $\varphi(x)$: 
\begin{equation}\label{eq:ED-phi}
    (\forall x: \varphi(x))(\psi(x) \vee \neg \psi(x)) \to (\exists x:\varphi(x)) \psi(x) \vee (\forall x: \varphi(x))\neg \psi(x) \tag{\textsc{Omni}}
\end{equation}
The upshot, then, is an entirely non-modal analysis of extensional definiteness, namely, that $\varphi(x)$ is ED iff quantification restricted to $\varphi(x)$ behaves classically, in the sense expressed by (\ref{eq:ED-phi}).

Late in his life, Solomon Feferman proposed a nice slogan: ``What's definite is the domain of classical logic, what's not is that of intuitionistic logic" \cite[p. 23]{Feferman:Is_CH_def}. We now have a Weyl-inspired motivation for this slogan: we understand Feferman's ``definite'' as Weyl's ``extensionally definite'' and analyze the latter in the way just explained. This entails that what is ``definite'' is indeed the domain of classical logic, while otherwise only intuitionistic logic prevails.\footnote{\cite[p. 16]{Feferman:2014-conceptualstrucutralism}'s own motivation for the slogan runs together extensional determinateness with the more demanding notion of completability. }

\subsection{Principles of definiteness}\label{sec:princ-def}

Our next aim is to establish some principles of definiteness. We begin by defining a weaker notion of \textit{intensional definiteness} (abbreviated \textit{ID}). We say that a property $F$ is ID iff $\forall x (Fx \vee \neg Fx)$.
Next, by permitting quantification over properties, we can provide an explicit definition of extensional definiteness. We say that $X$ is ED iff, for every ID $F$, we have: 
\begin{equation}\label{eq:ED-X}
    (\exists x:Xx) Fx \vee \neg (\exists x: Xx)Fx \notag 
\end{equation}
In other words, $X$ is extensionally definite iff  quantification restricted to $X$ preserves intensional definiteness.\footnote{The definition mentions only this preservation only for the existential quantifier. But it is easily derived for the universal as well.} Both definitions are easily extended to polyadic relations. 

Given these definitions, intuitionistic logic 
alone suffices to prove some surprisingly informative principles of defniteness:
\begin{Thm}\label{thm:Int-ED}
    Intuitionistic logic 
    proves the following principles of definiteness: 
\begin{enumerate}[(a)]
     \item 
    ID conditions are closed under quantifier-free definitions. 
     \item ID conditions are closed under ED-bounded quantification 
    \item \textit{Separation}: Assume $X$ is ED and $F$ is ID. Then the intersection of the two concepts, namely $\lambda \vec{x}. (X \vec{x} \wedge F\vec{x})$, is ED. 
    \item \textit{Generalized Union}: If $X$ is ED and for each $x$ such that $Xx$ the property $\lambda y.\psi(x,y)$ is ED, then the property $\lambda y \exists x (Xx \wedge \psi(x,y)$ is ED. 
        \item  \textit{Replacement}: If $X$ is ED and $\forall x \in X \exists ! y \, \psi(x,y)$, then  $\lambda y (\exists x \in X)\psi(x,y)$ is ED. 
\end{enumerate}
Assume that identity is ID. Then we additionally have: 
\begin{enumerate}[(a)]\setcounter{enumi}{5}
    \item Every ED concept is also ID. 
    \item \textit{Adjunction}: For any ED property $X$ and any $a$, the property $\lambda x (Xx \vee x = a)$ is also ED. 
\end{enumerate}
   
\end{Thm}
See \cite{Brauer&Linnebo:InvariancePredSets} for a proof.

\vspace{4mm}

I observed 
that every plurality is ED. So it is unsurprising that every plural comprehension principle that is part of the intuitionistic plural logic I-BPL has an analogue concerning ED properties, namely, Empty, Adjunction,  Pairwise Union (a simple consequence of Generalized Union), and Decidable Separation. However, the principles of Theorem~\ref{thm:Int-ED} are better than these plural comprehension principles, in two ways. First, the comprehension principles established for ED properties are stronger than those adopted in I-BPL for pluralities. It is easily shown that (plural analogues of) Generalized Union and Replacement do not follow from I-BPL. 
Second, the comprehension principles for ED properties follow from intuitionistic logic and our definitions---and, in some cases, also ID(=)---whereas all the plural comprehension principle of I-BPL have to be assumed as axioms. I find it amazing that logic and definitions should give rise to such rich principles. 


\subsection{Predicative set theory}\label{sec:pred-sets}

We can now formulate a Weyl-inspired approach to predicative set theory.\footnote{This subsection summarizes a more thorough discussion found in \cite{Brauer&Linnebo:InvariancePredSets}.} Stated in prose, the theory is remarkably simple: 
\begin{quote}
    The property of being a natural number 
    is ED. Every ED property has a set as its extension. This is the only way to form sets. 
\end{quote}

More formally, we first lay down the ED-ness of $\mathbb{N}$ as an omniscience principle. Next, we adopt the principle that every ED property defines a set. 
\begin{equation}\label{eq:ED-Set}
    \forall F (\ed(F) \to \exists y \forall x (x \in y \leftrightarrow Fx) \tag{ED$\to$S}
\end{equation}
The idea is that any property whose instances have been properly ``demarcated'' is suitable for defining a set. Then, we lay down Extensionality: 
\begin{equation*}\label{eq:Ext}
    \set(xx,x) \wedge \set(yy,y) \to (x=y \leftrightarrow \forall z(z \prec xx \leftrightarrow z \prec yy)) \tag{\textsc{Ext}}
\end{equation*}
Finally, we use an induction principle to express that there is no way to form sets other than by means of ED concepts. Let $\Her(\varphi)$ abbreviate the statement that $\varphi$ is inherited from the instances of any ED property to the resulting set, that is, that for any ED $F$ we have $\forall x (Fx \to \varphi(x)) \to \varphi(\{u: Fu\})$. Then our axiom schema of induction is:
\begin{quote}
    Suppose every non-set is $\varphi$ and $\Her(\varphi)$. Then $\forall x \varphi(x)$.
\end{quote}

Taken together, these simple axioms yield a natural and fairly strong constructive theory of predicative sets, \textit{cPS}. In \cite{Brauer&Linnebo:InvariancePredSets} this theory is shown to be closely related to Feferman's semi-constructive set theory SCS. In particular, the two theories have the same proof-theoretic strength, namely that of Kripke Platek set theory---and thus well beyond the traditional ``limit of predicativity''. 

This attractive approach to predicative set theory would not be available without strict potentialism and its use of intuitionistic logic (when both modalities are switched off) or of two distinct modalities 
(when both modalities are switched on). In other words, the approach requires that the universe not be ED. 

In fact, we can say something far stronger: the approach requires that the property of being a \textit{subset} of $\mathbb{N}$ not  be ED. For suppose this property were ED. Then every condition in the language of second-order arithmetic---even impredicative ones---would be ED and thus define sets.\footnote{We see this as follows. The atomic predicates of arithmetic are ID. If both the properties $\mathbb{N}$ and $\lambda x. x \subseteq \mathbb{N}$ are ED, every condition in the language of second-order arithmetic would be ID, and thus, by Separation, also ED.} Since this would obviously be unacceptable to predicativists, we have to deny omniscience for subsets of~$\mathbb{N}$: 
\begin{equation}\label{eq:subP-OmniN}
    (\forall X \subseteq \mathbb{N})(\varphi(X) \vee \neg \varphi(X)) \to (\exists X \subseteq \mathbb{N})\varphi(X) \vee \neg (\exists X \subseteq \mathbb{N}) \varphi(X) \tag{\textsc{Omni}-$\subseteq\!\mathbb{N}$} 
\end{equation}
Switching both modalities back on, the failure of the omniscience principle (\ref{eq:subP-OmniN}) corresponds to the failure of: 
\begin{equation} \label{eq:inext-sub-ED}
    \Box_G (\forall X : X \subseteq \mathbb{N}) \Box_D \theta \to \Box_D (\forall X : X \subseteq \mathbb{N})  \theta
\end{equation}
In words: given any ways to generate subsets of $\mathbb{N}$, we cannot conclude that these are \textit{determinately} all the ways to generate such subsets. 
On the contrary, any ways to generate subsets of $\mathbb{N}$ can be used to specify yet other way to generate such subsets. 

I observed in Section~\ref{sec:amount-indet} that the distance between G- and D-modality provides a measure of some strict potentialist theory's non-classicality. And indeed, our predicative set theory holds that, as concerns natural numbers, the two modalities coincide, but not so as concerns sets of numbers.  (These commitments correspond to the acceptance of  (\ref{eq:ED-N})) and rejection of (\ref{eq:inext-sub-ED}), respectively.)

\section{Cantor's domain principle} \label{sec:DP}

A second application of my general analysis of strict potentialism concerns Cantor's famous \textit{domain principle}, which states that every potential infinity presupposes a corresponding actual infinity. 
\begin{quote}
    In order for there to be a variable quantity in some mathematical study, the ``domain'' of its variability must strictly speaking be known beforehand through a definition. However, this domain cannot itself be something variable, since otherwise each fixed support for the study would collapse. Thus, this ``domain'' is a definite, actually infinite set of values. Thus, each potential infinite, if it is rigorously applicable mathematically, presupposes an actual infinite. \cite[p.~9]{Cantor:1886}, transl. in \cite[p.~25]{Hallett:1984}
\end{quote}
The idea is that, for an infinite domain of ``variability'' (or quantification) to be properly defined, it must correspond to ``a definite, actually infinite, set''. 

\subsection{When the domain principle \textit{might} apply}

Is the domain principle justified? 
There are two cases to consider. 

First, suppose the domain in question is not even ED. Such a domain lacks the precise specification that Cantor seeks. It is extensionally \textit{in}definite precisely because it is  \textit{not} ``known beforehand through a definition''. For that reason, the ordinary, instantial conception of quantification is not available for such domains.\footnote{For the distinction between instantial and non-instantial conceptions of quantification, see \cite{Linnebo:GenExpl}. } Rather, quantification over a non-ED domain must---and can---be understood in terms of non-instantial generality, which requires intuitionistic logic. This provides a form of ``fixed support for the study'' that  was unavailable at Cantor's time. This new fixed support enables us to reject the domain principle as applied to a non-ED domain. 

Second, suppose the domain in question is ED. Weyl would insist, as against Cantor, that an ED domain, such as that of the natural numbers, can be well defined---and even be ``definite'' in a fairly robust sense---without a corresponding actually infinite set. I believe the analysis in the previous section proves Weyl right: we can make perfectly good sense of an domain that is ED but not completable. Thus, even applied to an ED domain, the domain principle is not \textit{conceptually required}. 

It might still be interesting, however, to explore the domain principle as a \textit{hypothesis}. We can regard it as an idealization of the sort that has so fruitfully been applied in Cantorian set theory.\footnote{There are ``revenge arguments'' that resemble the domain principle; cf. \cite[Sect.~7.5]{Studd:Everything-book} and \cite{Studd:Linnebo-BCP}. 
Since the Cantorian idealizations are appropriate in the context of such arguments, the second, Weylian response is not very robust. Rather, the best way to resist ``the revenge'' is the first response, i.e., to embrace strict potentialism and deny that the domain of all sets is ED. See also \cite[Sect.~9]{Linnebo:2023-Theoria-replies}.} In our framework, with both modalities switched off, this idealization takes the form: 
\begin{equation}\label{eq:DP}
    \forall F (\ed(F) \to \exists xx \forall y (y \prec xx \leftrightarrow Fy) \tag{ED$\to$P}
\end{equation}
In other words, we are supposing that the \textit{prima facie} less demanding notion of extensional definiteness suffices for the \textit{prima facie} more demanding notion of completability. 

\subsection{The cash value of completability}\label{sec:rigidityprinciples}

What, exactly, is gained in this passage from an ED property to an associated plurality? I have stressed that whereas every ED property is demarcated, a plurality is not only demarcated but also completed. What is the logical cash value of this conceptual claim? 

The answer concerns rigidity principles, in the sense of Section~\ref{sec:modallogicplurals}. Since a plurality $xx$ is completed, each member of $xx$ is available wherever $xx$ are available. Likewise, any objects $yy$ among $xx$ are available wherever $xx$ are available. This is why 
we adopted the inextendability of the notion of being a subplurality: 
\begin{equation}
    \forall xx (xx \preccurlyeq yy \to \Box_D \theta) \to \Box_D \forall xx (xx \preccurlyeq yy \to \theta)\tag{\textsc{InExt}-$\preccurlyeq$}
\end{equation}
Along the lines of Lemma~\ref{Lem:godel-omni}, this inextendability has an intuitionistic correlate, namely the following omniscience property: 
\begin{equation}\label{eq:subP-Omni}
    (\forall xx \preccurlyeq aa)(\varphi(xx) \vee \neg \varphi(xx)) \to (\exists xx \preccurlyeq aa)\varphi(xx) \vee \neg (\exists xx \preccurlyeq aa) \varphi(xx) \tag{\textsc{Omni}$\preccurlyeq$}
\end{equation}

There is a sharp contrast here with predicativism. Recall that Weyl and other predicativists regard the natural numbers $\mathbb{N}$ as ED but deny omniscience for its subsets: 
\begin{equation}
    (\forall X \subseteq \mathbb{N})(\varphi(X) \vee \neg \varphi(X)) \to (\exists X \subseteq \mathbb{N})\varphi(X) \vee \neg (\exists X \subseteq \mathbb{N}) \varphi(X) \tag{\textsc{Omni}-$\subseteq\!\mathbb{N}$} 
\end{equation}
Thus, the domain principle in the form (\ref{eq:DP}) has the substantial consequence of ensuring omniscience for subsets of the natural numbers---and, more generally, of the subsets of any ED domain.

\subsection{A new route to intuitionistic plural logic}\label{sec:newroute-IPL}

The domain principle provides a new route to intuitionistic plural logic. First, we prove the principles of definiteness stated in Theorem~\ref{thm:Int-ED}. Then, by invoking (\ref{eq:DP}), these principles carry over to analogous principles concerning the existence of pluralities. 

This route is quite attractive. For one thing, we obtain a stronger intuitionistic plural logic than the logic I-BLP that we obtained, in Section~\ref{sec:bi-mirroring-PL}, via plural mirroring. This stronger logic includes plural comprehension principles as strong as the ED comprehension principles of Theorem~\ref{thm:Int-ED}. As noted, this takes us beyond I-BPL by adding plural analogues of Generalized Union and Replacement.\footnote{For  \textit{cognoscenti}, this is an intuitionistic version of the Critical Plural Logic of \cite[ch.~12]{Florio&Linnebo-book}.} For another, this stronger plural logic has a single idea behind it, namely, that every ED property defines a plurality. And as already proclaimed, the principles of ED-ness follow from intuitionistic logic and definitions alone, albeit in two cases, we additionally require ID(=).

To summarize, we have two further examples of the value of strict potentialism and our bimodal analysis of it. We obtain an interesting analysis of Cantor's domain principle and the conditions under which this principle might be acceptable. Further, we obtain a new route to a strong---but still non-traditional---intuitionistic plural logic that would otherwise not be available. Both advances hinge on the concept of extensional definiteness, which is only available to strict potentialists and admits of a useful bimodal gloss. Without strict potentialism, the universe would be ED by default and our principles of extensional definiteness and plural existence would trivialize. 


\section{Strict potentialism about Cantorian sets} \label{sec:strict-pot-ST} 

As a final illustration of the analysis developed in this article, I will outline a strict potentialist version of Cantorian set theory. 

\subsection{The core theory}

Just like the predicativist set theory, this theory has a remarkably simple statement in prose: 

\begin{quote}
    There is a plurality of all natural numbers. Every plurality forms a set. This is the only way to form sets. 
\end{quote}

Let us begin with the idea that any available objects form a set. In the bimodal setting, this idea is formalized as: 
\begin{equation*}\label{eq:collapse}
    \Box_D \forall xx \, \Diamond_G \exists y \, \set(xx,y) \tag{Collapse$^\ast$}
\end{equation*}
Turning both modalities off, this yields: 
\begin{equation}\label{eq:P-S}
    \forall xx \, \exists y \, \set(xx,y) \tag{P$\to$S}
\end{equation}
That is, any completed lot of objects---or, as Cantor would put it, any ``consistent multiplicity''---defines a set. Next, we formalize the claim that all sets are generated in this way, in a well-founded manner, by adopting the following induction scheme: 
\begin{quote}
    Suppose that every urelement is $\varphi$ and that, for any $xx$ each of which is $\varphi$, $\{xx\}$ too is $\varphi$.
    Then everything is $\varphi$.
\end{quote}

\subsection{What pluralities are there?}

It becomes harder, but more interesting, when we ask what pluralities there are. We need a plural logic that provides information about what objects are jointly available to be collected into a set. The traditional option would be to start with the plural logic I-BPL that we obtain via the bimodal mirroring of Theorem~\ref{thm:BM-PL}, and then to strengthen this plural logic in various ways so as to obtain the set theoretic principles that we want.\footnote{This would be a strict potentialist analogue of the approach of \cite{Linnebo:2013-PHS}, cf. also \cite{Studd_ItConceptBiModal}. This approach is sketched in the appendix of \cite{Linnebo&Shapiro:2024-PredClassesStrictPot}.} 

Here I will adopt an alternative approach, which better exploits the opportunities that open up in the strict potentialist setting. We begin with minimal plural logic, which has no comprehension principles, and the plural rigidity principles.\footnote{Note that this adds two principles that are not available for ED properties, namely, Plural Choice and the strong rigidity principle (\ref{eq:subP-Omni}), which we saw in Section~\ref{sec:rigidityprinciples} to be distinctive of pluralities.} To obtain the needed plural comprehension principles, we use the new route to intuitionistic plural logic described in Section~\ref{sec:newroute-IPL}. To recall, this route has two steps: 
\begin{itemize}
    \item Using just intuitionistic logic and definitions, and in some cases ID(=), we prove the principles of intensional and extensional definiteness of Theorem~\ref{thm:Int-ED}. 
    \item We adopt Cantor's domain principle, in the form (\ref{eq:DP}), to allow every ED property to define a plurality. That is, we assume that every demarcated property can be completed.\footnote{Since every plurality defines an ED property and \textit{vice versa}, it would be possible to streamline the account by using just a single style of higher-order variable: plural quantification can be eliminated in favor of quantification over ED properties. For reasons of simplicity, though, I here forego this notational economy.} 
\end{itemize}
Finally, we lay down that there is a plurality of all natural numbers. 



\subsection{Deriving axioms of ordinary set theory}

I wish to observe how two distinctive axioms of transfinite set theory, which ordinarily require further assumptions, emerge from our theory naturally and for free. 

One example is the set-theoretic axiom of Replacement. This follows immediately from Plural Replacement (or, for that matter, Generalized Union). 
This contrasts favorably with the traditional approach via mirroring, where an additional assumption is needed to prove Replacement.\footnote{Specifically, we assume $(\forall x \prec xx) \exists!  y \, \psi(x,y) \to \exists yy (\forall x \prec xx) (\exists  y \prec yy) \, \psi(x,y)$ or its modalized analogue.} 

Another example is the Powerset axiom. Observe first that 
we can derive omniscience for subsets of any given set:\footnote{This follows from omniscience for subpluralities of any given plurality, (\ref{eq:subP-Omni}), along with (\ref{eq:P-S}).  }
\begin{equation}\label{eq:subS-Omni}
    (\forall x \subseteq a)(\varphi(x) \vee \neg \varphi(x)) \to (\exists x \subseteq a)\varphi(x) \vee \neg (\exists x \subseteq a) \varphi(x) \tag{\textsc{Omni}-$\subseteq$}
\end{equation}
Next, applying the domain principle, (\ref{eq:DP}), this yields a plurality of all subsets of $a$. Finally, by (\ref{eq:P-S}), we obtain the powerset of $a$. Again, this contrasts favorably with the traditional approach via mirroring, where an additional assumption is needed to prove Powerset.\footnote{We need to assume that, for any set $x$, there is a plurality of all of its subsets:
$\forall x \, \exists yy \,  \forall z (z \prec yy \leftrightarrow z \subseteq x)$.}

Summing up, we obtain a strict potentialist set theory much like ZFC except:\footnote{The theory is closely related to Feferman's Semi-Constructive Set theory SCS, supplemented with Powerset. See the appendix of \cite{Linnebo&Shapiro:2024-PredClassesStrictPot} for a comparison.} 
\begin{enumerate}[(i)]
        \item the logic is semi-intuitionistic (in the sense that we have intuitionistic logic, strengthened with decidability of identity and $\in$ as well as Bounded Omniscience; 
        \item instead of full Separation we have Separation only for decidable conditions
        \item instead of Foundation we have the classically equivalent principle of $\in$-induction. 
    \end{enumerate} 

This provides yet another example of a failure of Reverse Subsumption. If this principle were assumed, the logic would go classical. This would make every property ED, which in turn would render the theory inconsistent. Putting on our bimodal spectacles, the failure of Reverse Subsuption means that, whenever we have accepted some ways to generate sets, we can use these to specify yet stronger ways to generate sets.\footnote{One version of this phenomenon is the ``set revenge argument'' from \cite{Linnebo&Shapiro:2024-PredClassesStrictPot}.}

\section{Concluding summary}

Potentialism holds that objects are generated in an incompletable process. Strict potentialism adds that truths are successively made true. In the opening three sections, I used two modalities to explicate strict potentialism: one for the generation of objects (G-modality) and another for statements being determined as true (D-modality). This gave rise to a classical bimodal logic BM. 

I proceeded to investigate 
how one or both modalities can be ``switched'' on or off. When switched off, a modality is absorbed into the interpretation of the connectives and quantifiers. I identified the non-classical logics to which this ``switching off'' or ``demodalization'' gives rise: for G-modality, a restricted plural logic, and for D-modality, a semi-intuitionistic logic. To describe all the possibilities, I identified systems of plural logic that differ along several axes, namely, by
\begin{itemize}
    \item being based on classical or intuitionistic logic;
    \item being modal (including the modal rigidity axioms) or non-modal; 
    \item the amount of plural comprehension that is accepted, which could be Basic, Critical, or Traditional.
\end{itemize}
Using these systems, I established mirroring theorems that describe the results of all the possible ways to switch the two modalities on or off. 

In the final third of the article, I illustrated the value of the analysis by applying it to some important views in the foundations of mathematics. 
In addition to illuminating these view, this discussion revealed several reasons to accept strict, as opposed to liberal, potentialism: 
\begin{enumerate}[(i)]
    \item Weyl's notion of extensional definiteness becomes available;
    \item we can prove, using intuitionistic logic (nearly) alone, various useful principles of definiteness; 
    \item these principles enable a simple and natural, but fairly strong predicative set theory; 
    \item we can formulate a plausible explication of Cantor's famous domain principle; 
    \item this explication provides a unified route to a strong intuitionistic plural logic; 
    \item this plural logic enables a remarkably simple formulation of strict potentialism about sets, where Replacement and Powerset can be proved without the need for any additional assumptions. 
\end{enumerate}
None of these results would have been  available had we accepted only liberal, but not strict, potentialism. 

I submit that it is useful to be able to look at these results through both modal and non-modal glasses. The (bi)modal perspective is typically more conceptually illuminating, whereas the non-modal approaches are considerably simpler and more user-friendly. 

\appendix
\section{Proof sketch for Theorem~\ref{thm:fullstory-plural}}

We start with (a) (i.e., the potentialist mirroring), upper row, preservation of theoremhood left to right. We already have the mirroring of the underlying first-order systems. Next, we add Minimal Plural Logic on both sides. Then, the plural rigidity axioms added on the left are also added on the right. Finally, following the strategy of \cite[Appendix~A.1]{Linnebo:PotDemodal} for the proof of classical plural mirroring (Theorem~\ref{thm:plural-mirroring} above), it is routine to show that the axioms of decidable BPL are mapped to theorems on the right. 

For the reverse direction, we modify the reverse translation defined in (\textit{op. cit}., Appendix~A.2) 
to treat $\Box_G$ and $\Diamond_G$ separately:
\begin{gather*}
    [\Box_G \varphi]_{xx} = (\forall yy \succcurlyeq xx) [\varphi]_{yy}  \\
    [\Diamond_G \varphi]_{xx} = (\exists yy \succcurlyeq xx) [\varphi]_{yy} 
\end{gather*}
The proof from the classical analogue of the desired reverse direction is fully constructive and thus carries over to our present setting. 

We turn to the lower row, preservation of theoremhood left to right. We begin with the mirroring of the underlying first-order systems. Minimal Plural Logic is then added on both sides. The plural rigidity axioms added on the left are also present on the right. Finally, the plural comprehension principles added on the left are mapped to theorems on the right, as can again be seen by reasoning as in (\textit{op. cit}., Appendix~A.1). 

For the reverse direction, we extend the reverse translation of (\textit{op. cit}., Appendix~A.2) by adding the following clause: 
\begin{equation*}
    [\Box_D \varphi]_{xx} = \Box (\forall yy \succcurlyeq xx) [\varphi]_{yy}
\end{equation*}
Clearly, the reverse translation of (\ref{eq:subsump})) is a theorem. The reverse translation of \ref{eq:mixed-G} is: 
\begin{equation*}
    (\exists xx \succcurlyeq aa)\Box(\forall yy \succcurlyeq xx) [\varphi]_{yy} \to \Box (\forall xx \succcurlyeq aa)(\exists yy \succcurlyeq xx)[\varphi]_{yy}
\end{equation*}
which is a theorem.


We turn now to (b), which concerns the G\"odel mirroring. Consider first the left-hand column, preservation of theoremhood from top to bottom. Building on Theorem~\ref{thm:mirroring-G}, it suffices to consider the axioms involving plurals. The plural rigidity axioms are mapped to theorems by Lemmas~\ref{Lem:godel-prec} and \ref{Lem:godel-omni}. It remains to consider the axioms of decidable BPL. An illustrative case is (\ref{eq:Dec-P-Sep}), which translates as: 
\begin{equation}
    (\forall x \prec xx))(\varphi^g(x) \vee \Box \neg \varphi^g(x)) \to \exists yy \Box \forall x (x \prec yy \leftrightarrow x \prec xx \wedge \varphi^g(x)) 
\end{equation}
This is a theorem, by reasoning as in (\textit{op. cit}., Appendix~A.1). The other comprehension axioms are similar. 


Let us turn, finally, to the right-hand column, i.e., (c). Using the techniques developed above, it is routine to prove the preservation of theoremhood from top to bottom. We have already observed that the reverse direction fails, that is, that the G\"odel embedding on the right is not faithful. $\dashv$

\end{spacing}

\begin{footnotesize}
\bibliographystyle{apalike}
\bibliography{Total}

@Preamble{ " \newcommand{\noop}[1]{} " }

@article{Brauer-Linnebo-Shapiro:DivPot,
	title = {Divergent Potentialism: {A} Modal Analysis With an Application to Choice Sequences},
	doi = {10.1093/philmat/nkab031},
	journal = {Philosophia Mathematica},
	year = {2022},
	author = {Ethan Brauer and {\O}ystein Linnebo and Stewart Shapiro}
}

@ARTICLE{Crosilla-Linnebo:2024-Weyl,
  author =       {Laura Crosilla and {\O}ystein Linnebo},
  title =        {Weyl and Two Kinds of Potential Domains},
  journal =      {No\^{u}s},
  year =         {2024},
  volume =       {58},
  number =       {2},
  pages =        {409-30},
  doi =          {doi.org/10.1111/nous.12457},
}

@BOOK{Ewald:1996,
  AUTHOR =       {William Ewald},
  TITLE =        {From Kant to Hilbert: A Source Book in the Foundations of Mathematics},
  PUBLISHER =    {Oxford University Press},
  YEAR =         {1996},
  volume =       {2},
  address =      {Oxford},
}

@UNPUBLISHED{Feferman:Is_CH_def,
  author =       {Solomon Feferman},
  title =        {Is the continuum hypothesis a definite mathematical problem?},
  note =         {Unpublished manuscript},
  year =         {2011},
  howpublished = {\url{http://logic.harvard.edu/EFI Feferman IsCHdenite.pdf}},
}

@INCOLLECTION{Feferman:2014-conceptualstrucutralism,
  AUTHOR =       {Solomon Feferman},
  TITLE =        {Logic, Mathematics and Conceptual Structuralism},
  BOOKTITLE =    {The Metaphysics of Logic},
  PUBLISHER =    {Cambridge University Press},
  YEAR =         {2014},
  editor =       {Penelope Rush},
  pages =        {72--92},
  address =      {Cambridge},
}

@article{Florio&Linnebo-HO-Henkin,
	author = {Salvatore Florio and {\O}ystein Linnebo},
	doi = {10.1111/nous.12091},
	journal = {No\^{u}s},
	number = {3},
	pages = {565--583},
	title = {On the Innocence and Determinacy of Plural Quantification},
	volume = {50},
	year = {2016}
}

@ARTICLE{RasiowaSikorski:1953,
  AUTHOR =       {H. Rasiowa and R. Sikorski},
  TITLE =        {An Algebraic Treatment of the Notion of Satisfiability},
  JOURNAL =      {Fundamenta Mathematicae},
  YEAR =         {1953},
  volume =       {40},
  number =       {},
  pages =        {62-95},
}

@ARTICLE{Linnebo:PotDemodal,
  AUTHOR =       {{\O}ystein Linnebo },
  TITLE =        {Potentialism Demodalized},
  YEAR =         {2026},
  journal =      {Review of Symbolic Logic},
  volume =       {?},
  number =       {?},
  pages =        {?},
  note = {doi:10.1017/S1755020326101130},
  doi =          {doi:10.1017/S1755020326101130},
}

@ARTICLE{Cantor:1886,
  AUTHOR =       {Georg Cantor},
  TITLE =        {{\"U}ber die Verschiedenen {A}nsichten in {B}ezug auf die atualunendlichen {Z}ahlen},
  JOURNAL =      {Bihang till Kongl. Svenska vetenskaps-akademiens handlingar},
  YEAR =         {1886},
  volume =       {11},
  issue =         {19},
  pages =        {1-10},
}

@MISC{Brauer&Linnebo:InvariancePredSets,
  AUTHOR =       {Ethan Brauer and {\O}ystein Linnebo},
  TITLE =        {Predicativity as Invariance: the Case of Sets},
  YEAR =         {2025},
  howpublished = {Typescript},
}

@MISC{Linnebo&Litland:GenGrounded,
  AUTHOR =       {{\O}ystein Linnebo and Jon Erling Litland},
  TITLE =        {Grounding generalities},
  YEAR =         {2025},
  howpublished = {Unpublished manuscript},
}

@BOOK{Florio&Linnebo-book,
  author =       {Salvatore Florio and {\O}ystein Linnebo},
  TITLE =        {The Many and the One: A Philosophical Study of Plural Logic},
  PUBLISHER =    {Oxford University Press},
  YEAR =         {2021},
  address =      {Oxford},
}

@BOOK{Hallett:1984,
  AUTHOR =       {Michael Hallett},
  TITLE =        {Cantorian Set Theory and Limitation of Size},
  PUBLISHER =    {Clarendon},
  YEAR =         {1984},
  address =      {Oxford},
}

@BOOK{Linnebo:2017-PhilMaths,
  author =       {{\O}ystein Linnebo},
  title =        {Philosophy of Mathematics},
  publisher =    {Princeton University Press},
  year =         {2017},
  address =      {Princeton, NJ},
}

@ARTICLE{Linnebo:GenExpl,
  author =       {{\O}ystein Linnebo},
  title =        {Generality Explained},
  journal =      {Journal of Philosophy},
  year =         {2022},
  volume =       {119},
  number =       {7},
  pages =        {349-379},
}

@ARTICLE{Linnebo&Shapiro:ActPotInf,
  author =       {{\O}ystein Linnebo and Stewart Shapiro},
  title =        {Actual and Potential Infinity},
  journal =      {No\^us},
  year =         {2019},
  volume =       {53},
  number =       {1},
  pages =        {160--191},
}

@article{Linnebo&Shapiro:2023-PredFormPot,
	year = {2023},
	author = {{\O}ystein Linnebo and Stewart Shapiro},
	pages = {1--32},
	journal = {Review of Symbolic Logic},
    volume = {16},
    number = {1},
	doi = {10.1017/s1755020321000423},
	title = {Predicativism as a Form of Potentialism}
}

@article{Linnebo&Shapiro:2024-PredClassesStrictPot,
	year = {2024},
	author = {{\O}ystein Linnebo and Stewart Shapiro},
	pages = {?},
	journal = {Philosophia Mathematica},
    volume = {?},
    number = {?},
	doi = {?},
	title = {Predicative Classes and Strict Potentialism}
}

@ARTICLE{Linnebo:2013-PHS,
  AUTHOR =       {{\O}ystein Linnebo},
  TITLE =        {The Potential Hierarchy of Sets},
  YEAR =         {2013},
  JOURNAL =      {Review of Symbolic Logic},
  volume =       {6},
  number =       {2},
  pages =        {205-228},
}

@article{Linnebo:2023-Theoria-replies,
	title = {Replies},
	doi = {10.1111/theo.12465},
	author = {{\O}ystein Linnebo},
	publisher = {Wiley},
	volume = {89},
	journal = {Theoria},
	number = {3},
	pages = {393--406},
	year = {2023}
}

@INCOLLECTION{Linnebo:2027-WhatIsPot,
  author =       {{\O}ystein Linnebo},
  title =        {What is Set-Theoretic Potentialism?},
  booktitle =    { Oxford Handbook of Philosophy of Set Theory},
  publisher =    {Oxford University Press},
  year =         {202x},
  editor =       {Luca Incurvati and Giorgio Venturi},
  pages =        {??},
  address =      {Oxford},
}

@PHDTHESIS{SimpsonAlex-int-modal-logic,
  author =       {Alex K. Simpson},
  title =        {The Proof Theory and Semantics of Intuitionistic Modal Logic},
  school =       {University of Edinburgh},
  year =         {1994},
  type =         {Ph{D} dissertation},
}

@article{Studd_ItConceptBiModal,
	title = {The Iterative Conception of Set: A (Bi-)Modal Axiomatisation},
	journal = {Journal of Philosophical Logic},
	year = {2013},
	pages = {697--725},
	volume = {42},
	number = {5},
	author = {James Studd}
}

@BOOK{Studd:Everything-book,
  author =       {James P. Studd},
  title =        {Everything, more or less: A Defence of Generality Relativism},
  publisher =    {Oxford University Press},
  year =         {2019},
  address =      {Oxford},
}

@article{Studd:Linnebo-BCP,
	author = {J. P. Studd},
	doi = {10.1111/theo.12356},
	journal = {Theoria},
	number = {3},
	pages = {366--392},
	publisher = {Wiley},
	title = {Linnebo's Abstractionism and the Bad Company Problem},
	volume = {89},
	year = {2023}
}

@BOOK{Weyl:1918,
  AUTHOR =       {Hermann Weyl},
  TITLE =        {Das Kontinuum},
  PUBLISHER =    {Verlag von Veit \& Comp},
  YEAR =         {1918},
  address =      {Leipzig},
  note =         {Translated as \emph{The Continuum} by S. Pollard and T. Bole, Dover, 1994.},
}

@ARTICLE{Weyl:Grundlagenkrise,
  AUTHOR =       {Hermann Weyl},
  TITLE =        {\"{U}ber die neue {G}rundlagenkrise der {M}athematik},
  JOURNAL =      {Mathematische Zeitschrift},
  YEAR =         {1921},
  volume =       {10},
  number =       {1--2},
  pages =        {39--79},
  note =         {Translated in  P. Mancosu (ed.),\emph{From Brouwer to Hilbert: The Debate on the Foundations of Mathematics in the 1920s},  (Oxford University Press, 1998)},
}
\end{footnotesize}

\end{document}